\documentclass[12pt,reqno]{amsart}
\usepackage{hyperref}
\usepackage{amsmath,amssymb}
\usepackage[english]{babel}
\usepackage{caption}
\usepackage{amssymb,amsmath,euscript,enumerate,tikz}
\usepackage[margin=1in]{geometry}
\usepackage{graphicx}
\usepackage{float}
\usepackage{pgf,tikz}
\usepackage{mathrsfs}
\usepackage{upgreek}
\newcommand{\sk}{{\ensuremath{\sf k }}}
\newcommand{\m}{\ensuremath{\mathfrak m}}

\DeclareMathOperator{\Tor}{Tor}

\newtheorem{conjeture}{ Conjecture}[section]
\newtheorem{theorem}[conjeture]{ Theorem}
\newtheorem{lemma}[conjeture]{ Lemma}
\newtheorem{corollary}[conjeture]{ Corollary}
\newtheorem{proposition}[conjeture]{ Proposition}
\theoremstyle{definition}
\newtheorem{remark}[conjeture]{ Remark}

\newtheorem{notation}[conjeture]{ Notation}

\providecommand\ass{\text{\rm Ass}}
\providecommand\codim{\text{\rm codim}}

\providecommand\depth{\text{\rm depth}}
\renewcommand\dim{\text{\rm dim}}

\providecommand\pd{\text{\rm pd}}
\providecommand\hgt{\text{\rm ht}}
\providecommand\Minh{\text{\rm Minh}}
\providecommand\max{\text{\rm max}}

\providecommand{\Z}{{\mathbb Z}}

\providecommand\depth{{\rm depth}}
\providecommand\Tor{{\rm Tor}}
\providecommand\reg{{\rm reg}}

\begin{document}
\title{Certain Algebraic Invariants of Edge Ideals of Join of Graphs}
\author[Arvind Kumar]{Arvind Kumar$^1$}
\email{arvkumar11@gmail.com}
\thanks{$^1$ The author is funded by National Board for Higher Mathematics, India}
\address{Department of Mathematics, Indian Institute of Technology
    Madras, Chennai, INDIA - 600036}
\author[Rajiv Kumar]{Rajiv Kumar}
\email{gargrajiv00@gmail.com}
\address{Department of Mathematics, The LNM Institute of Information Technology, Jaipur, Rajasthan, INDIA - 302031}
\author[Rajib Sarkar]{Rajib Sarkar$^2$}
\email{rajib.sarkar63@gmail.com}
\thanks{$^2$ The author is funded by University Grants Commission,
    India}
\address{Department of Mathematics, Indian Institute of Technology
    Madras, Chennai, INDIA - 600036}
\begin{abstract}
    Let $G$ be a simple graph and $I(G)$ be its edge ideal. In this article, we study the Castelnuovo-Mumford regularity of symbolic powers of edge ideals of join of graphs. As a consequence, we prove Minh's conjecture for wheel graphs, complete multipartite graphs, and a subclass of co-chordal graphs. We obtain a class of graphs whose edge ideals have regularity three. By constructing graphs, we prove that the multiplicity of edge ideals of graphs is independent from the depth, dimension, regularity, and degree of $h$-polynomial. Also, we demonstrate that the depth of edge ideals of graphs is independent from the regularity and degree of $h$-polynomial by constructing graphs. 
\end{abstract}


\maketitle
\section{Introduction}
Let $G$ be a simple graph on the vertex set $V(G)=\{x_1,\dots,x_n\}$ and edge set $E(G)$. 
Set $S=\sk[x_1,\dots,x_n]$, where $\sk$ is an arbitrary field. The \textit{edge ideal} of $G$,
denoted by $I(G)$, is defined as $I(G):=(x_ix_j:\{x_i,x_j\}\in E(G))$. From the past one decade, many authors
established connections between the combinatorial properties of $G$ and the algebraic properties of $I(G)$. 
The $s$-th \textit{symbolic power} of $I(G)$ is defined as follows:
\[I(G)^{(s)}=\bigcap\limits_{\mathfrak{p}\in \ass(I(G))}\mathfrak{p}^s. \]
Let $I$ be a homogeneous ideal in $S$. The \textit{regularity} of $I$, denoted by $\reg( I)$, is defined as $\reg( I):=\max \{j: \Tor_{i}^S(I,\sk)_{i+j}\neq 0, \; i \geq 0 \}$. It is known that $\reg(S/I)= \reg(I)-1$.
Minh stated the following conjecture for the regularity of symbolic power of edge ideal, see \cite{GHJsymbo}:
\begin{conjeture}\label{conj}
    Let $G$ be a finite simple graph and $I(G)$ be its edge ideal. Then, $$\reg (I(G)^{(s)})=\reg(I(G)^s) \text{ for all } s\geq 1.$$ 
\end{conjeture}
Gu et al. studied Conjecture \ref{conj} for odd cycles in \cite{GHJsymbo}. In \cite{JK}, Jayanthan and Kumar
settled Conjecture \ref{conj} for  graphs, which are clique sum of an odd cycle and some bipartite graphs.
In \cite{Seyed-chordal} and \cite{FaUni}, Fakhari studied Conjecture \ref{conj} for chordal graphs and unicyclic graphs, respectively. In \cite{Se}, Selvaraja studied the regularity of $I(G)^s$ when $G$ is join of graphs. In this article, we study the regularity of $I(G)^{(s)}$, where $G$ is join of graphs. As a consequence, we prove Conjecture \ref{conj}
for some classes of graphs, for example, wheel graphs, complete multipartite graphs, and a subclass of co-chordal graphs.

In \cite{HMAedge}, Hibi et al. studied the relationship between regularity, degree of $h$-polynomial, and the number of
vertices. For a given pair of positive integers $(r,s)$, they constructed a graph whose edge ideal has regularity
$r$ and degree of $h$-polynomial $s$. In \cite{HKMext}, Hibi et al. constructed a graph for a given pair $(b,r)$ with $b\leq r$, 
where $b$ is the number of extremal Betti numbers of $S/I(G)$ and $r=\reg(S/I(G))$. In \cite{Gim-14}, Ramos and Gimenez characterized bipartite graphs whose edge ideals have regularity $3$. In this article, we obtain a class of graphs whose edge ideals have regularity three. Moreover, for a given pair of positive integers $(r,d)$ with $r\leq d$, we construct a graph $G$ such that $\reg(S/I(G))=r$ and $\dim(S/I(G))=d$. Also, we prove that the multiplicity of edge ideal of a graph has no relation with dimension, depth, regularity, and degree of $h$-polynomial. It is known that the depth and degree of $h$-polynomial for a square-free monomial ideal are bounded above by the dimension. We show that the depth of edge ideal of a graph is independent from the degree of $h$-polynomial and regularity.

\vskip 2mm
\noindent
\textbf{Acknowledgement:} The authors are grateful to A. V. Jayanthan for his
constant support and valuable comments. We want to express our
sincere thanks to the referees for a careful reading of the manuscript and raising crucial
points.
\section{Preliminaries}
In this section, we collect all the notation and definitions which are used throughout this article. For the undefined
terminology, we refer the readers to the book \cite{Mon} by Herzog and Hibi.
\subsection{Combinatorics}
Let $G$ be a simple graph on the vertex set $V(G)$ and edge set $E(G)$.
\begin{enumerate}[\rm a)]
    \item A graph $H$ is called an \emph{induced subgraph} of $G$ if $V(H) \subset V(G)$ 
    and for $i,j\in V(H)$, $\{i,j\}\in E(H)$ if and only if $\{i,j\}\in E(G)$.
    \item The \textit{complement} of a graph $G$, denoted by $G^c$, is the graph on the vertex set $V(G)$ 
    and edge set $\{ \{i,j\}:\{i,j\}\notin E(G) \text{ and } i \neq j \}$.
    \item A graph $G$ is \emph{chordal} if there is no induced cycle of length $\geq 4$ and it is \textit{co-chordal}
    if $G^c$ is chordal.
    \item A collection $C \subset E(G)$ of disjoint edges is said to be an  \textit{induced matching} if the induced 
    subgraph on the vertices of $C$ has no edge other than the edges in $C$. The maximum size of an induced matching 
    in $G$ is called the \textit{induced  matching number} of $G$ and is denoted by $\nu(G)$. 
    \item A collection $\Delta$ of subsets of $[n]$ is called a \textit{simplicial complex} if it satisfies the 
    following:
    \begin{enumerate}[\rm i)]
        \item for any $i\in [n]$, $\{i\}\in \Delta$,
        \item $F\subseteq G$ with $G\in \Delta$ implies that $F\in \Delta$.
    \end{enumerate}
    \item An element of $\Delta$ is called a \textit{face} and a maximal face with respect to  inclusion 
    is called a \textit{facet}.
    \item The \textit{dimension} of a simplicial complex  $\Delta$ is defined as $$\dim(\Delta)=\max \{|F|-1 : F \text{ is 
    a facet in } \Delta \} .$$
    \item Let $G$ and $H$ be graphs on vertex sets $V(G)=\{x_1,\dots, x_m\}$ and $V(H)=\{y_1,\dots, y_n\}$, 
    respectively. Then, the \emph{join} of $G$ and $H$, denoted by $G\ast H$, is a graph on the vertex set
    $V(G)\sqcup V(H)$ and edge set $E(G)\cup E(H)\cup\{\{x_i,y_j\}: 1\leq i\leq m, 1\leq j\leq n\}$.  The  graph shown in Figure \ref{join} is join of a vertex and cycle $C_5$. For a graph $G$ and $l\geq 1$, we denote $G^{\ast l}$ to be join of $l$-copies of $G$. In particular, $G^{\ast 1}=G$.
    \begin{figure}[H]
    \begin{tikzpicture}[scale =.8]
    \draw (0.98,-1.48)-- (0.04,0.34);
    \draw (3.58,-1.04)-- (0.98,-1.48);
    \draw (2.04,0.2)-- (3.58,-1.04);
    \draw (3.58,-1.04)-- (3.56,1.16);
    \draw (2.04,0.2)-- (3.56,1.16);
    \draw (2.04,0.2)-- (0.04,0.34);
    \draw (2.04,0.2)-- (0.98,-1.48);
    \draw (1.54,1.78)-- (0.04,0.34);
    \draw (3.56,1.16)-- (1.54,1.78);
    \draw (2.04,0.2)-- (1.54,1.78);
    \begin{scriptsize}
    \fill  (0.04,0.34) circle (1.5pt);
    \fill  (1.54,1.78) circle (1.5pt);
    \fill  (3.56,1.16) circle (1.5pt);
    \fill  (2.04,0.2) circle (1.5pt);
    \fill  (0.98,-1.48) circle (1.5pt);
    \fill  (3.58,-1.04) circle (1.5pt);
    \end{scriptsize}
    \end{tikzpicture}
    \caption{}\label{join}
    \end{figure}   
\end{enumerate}

\subsection{Hilbert Series}
Let $S=\sk[x_1,\dots,x_n]$ be a standard graded polynomial ring and $I\subseteq S$ be a homogeneous ideal with $\dim(S/I)=d$. By \cite[Theorem 6.1.3]{Mon}, the Hilbert series of $S/I$ is $H(S/I,t)=(h_0+h_1t+\dots+h_st^s)/(1-t)^d$ for $h_i\in \Z$ with $h_s\neq 0$.
\begin{enumerate}[\rm a)]
    \item The polynomial $h_0+h_1t+\dots+h_st^s$, denoted by $h_{S/I}(t)$, is called the $h$-polynomial of $S/I$.
    \item  The multiplicity of $S/I$, denoted by $e(S/I)$, is $h_{S/I}(1)$.
\end{enumerate}

\section{Regularity of symbolic powers of join of graphs} 
In this section, we study the regularity of symbolic powers of join of graphs. First, we prove a technical lemma, which is useful to prove our main result of this section.
\begin{lemma}\label{ArtinianReg}
    Let $S=\sk[x_1,\ldots,x_m,y_1,\ldots,y_n]$ be a polynomial ring.    Let $I\subset \langle x_1,\dots, x_m\rangle^2$ be a square-free monomial ideal in $\sk[x_1,\ldots,x_m]$, and $J\subset \langle y_1,\dots, y_n\rangle^2$ be  a square-free monomial  ideal in $\sk[y_1,\ldots,y_n]$. Then, $$\reg\left((I+\langle y_1,\dots, y_n\rangle)^{(s)}+(J+\langle x_1,\dots, x_m\rangle)^{(s)}\right)=2s-1, \text{ for all } s \geq 1.$$
\end{lemma}
\begin{proof}
    Let $K=(I+\langle y_1,\dots, y_n\rangle)^{(s)}+(J+\langle x_1,\dots, x_m\rangle)^{(s)}$. Since $x_i^s, y_j^s\in K$ for all $i,j$, $K$ is an $\m$-primary ideal, where $\m=\langle x_1,\ldots,x_m,y_1,\ldots,y_n\rangle$. Therefore, $$\reg(K)=\max\{j+1:(S/K)_j\neq 0\}.$$ 
    It is easy to see that $(S/K)_j=0$, for all $j\geq 2s-1$, we have $\reg(K) \leq 2s-1$. Since $I$ and $J$ are square-free monomial ideals, $x_1^{s-1}y_1^{s-1}\notin K$. Therefore, $x_1^{s-1}y_1^{s-1}\in (S/K)_{2s-2}$ is a nonzero element, and  the assertion follows.
\end{proof}
We now fix some notation which are used through out this section. For $1 \leq j \leq r$, set $X_j=\{x_{j1},\ldots,x_{jm_{j}}\}$. Let $G_1,\ldots,G_r$ be graphs on vertex sets $X_1,\ldots,X_r$, respectively. Also, for $1 \leq j \leq r$, set  $\m_j = \langle x_{j1},\ldots,x_{jm_j}\rangle$.
\begin{theorem}\label{symRegLem}
    For $r\geq 2$, let $G_1,\dots,G_r$ be graphs on vertex sets $X_1,\dots, X_r$, respectively. Let $G=G_1\ast\dots\ast G_r$ and $S=\sk[V(G)]$. Then, $$\reg(I(G)^{(s)})=\max\{\reg\left(I(G_j)^{(i)}\right)-i+s: 1 \leq i \leq s, \; 1 \leq j \leq r \}, \text{ for all } s \geq 1.$$
\end{theorem}
\begin{proof}
    We prove the assertion by induction on $r$. If $r=2$, then
    consider the following short exact sequence:
    $$
    0\rightarrow I(G)^{(s)}\rightarrow ( I(G_1)+\m_2)^{(s)}\oplus ( I(G_2)+\m_1)^{(s)}\rightarrow ( I(G_1)+\m_2)^{(s)}+( I(G_2)+\m_1)^{(s)}\rightarrow 0.
    $$
    It follows from \cite[Proposition 4.10]{HTT} that $$\reg(( I(G_2)+\m_1)^{(s)})=\max\{\reg(I(G_2)^{(i)})-i+s:1\leq i\leq s \} \geq 2s.$$ Similarly, $$\reg( ( I(G_1)+\m_2)^{(s)})=\max\{\reg(I(G_1)^{(i)})-i+s:1\leq i\leq s \}\geq 2s.$$
    By virtue of Lemma \ref{ArtinianReg}, $$\reg\left( ( I(G_1)+\m_2)^{(s)}+( I(G_2)+\m_1)^{(s)}\right)<\max\{\reg(( I(G_1)+\m_2)^{(s)}),\reg(( I(G_2)+\m_1)^{(s)}) \}, $$ and hence, it follows from the regularity behavior on short exact sequence that $$\reg\left(I(G)^{(s)}\right)=\max\left\{ \reg\left(( I(G_1)+\m_2)^{(s)}\right),\reg\left(
    ( I(G_2)+\m_1)^{(s)}\right)\right\}.$$
    Now, the assertion  follows from \cite[Proposition 4.10]{HTT}. Assume that $r>2$, and let $G'=G_1*\cdots *G_{r-1}$. Note that $G=G'*G_r$. Hence, the result follows from induction hypothesis and the base case.
\end{proof}
In \cite{Se}, Selvaraja defined following classes of graphs:
$$\mathcal{A}=\{G \mid \reg\left(I(G)^{s+1}:u\right)\leq \reg\left(I(G)\right), u\in \mathcal{G}(I(G)^s), s\geq 1 \}, $$ where $\mathcal{G}(I(G)^s)$ denotes the minimal monomial generating set of $I(G)^s$, and $$\mathcal{A}_1=\{G \mid \reg\left(I(G)^s\right)=2s+\reg(I(G))-2, s\geq 1\}.$$ 
It follows from \cite[Theorem 4.4, Remark 4.10]{Se} that $\mathcal{A}$ as well as $\mathcal{A}\cap\mathcal{A}_1$ are closed under the operation of join of graphs. We study the regularity of symbolic power of join of  graphs which belong to  class $\mathcal{A} \cap\mathcal{A}_1$. There are many classes of graphs which are contained in $\mathcal{A} \cap\mathcal{A}_1$.  For example, co-chordal graphs, unmixed bipartite, bipartite $P_6$-free graphs, forest graphs are contained in $\mathcal{A}\cap\mathcal{A}_1$, for more details see \cite{Se}. It follows from \cite[Corollary 4.6]{JS2018} that if $G$ is a chordal graph, then $G \in \mathcal{A}_1$. Also, it follows from the proof of \cite[Theorems 4.1, 4.4]{JS2018} that $G \in \mathcal{A}$ if $G$ is a chordal graph.
\begin{corollary}\label{regLemA1}
    Let $G_1,\dots ,G_r\in\mathcal{A}_1$ such that $G_i$'s are bipartite or chordal and $G=G_1\ast\dots\ast G_r$. Then, $\reg\left(I(G)^{(s)}\right) = 2s+\reg(I(G))-2, \text{ for all } s \geq 1.$ Furthermore, if $G_i\in \mathcal{A}$ for all $i$,  then $\reg\left(I(G)^{(s)}\right)=\reg(I(G)^s), \text{ for all } s \geq 1$. In particular, $\reg(I(G)^{(s)}) =2s$ if $G$ is a complete multipartite graph.
\end{corollary}
\begin{proof}
 If $G_i$ is  bipartite and $G_i \in \mathcal{A}_1$, then we have $$\reg\left(I(G_i)^{(s)}\right)=\reg(I(G_i)^s)<\reg\left(I(G_i)^{s+1}\right)=\reg\left(I(G_i)^{(s+1)}\right).$$ If $G_i$ is a chordal graph, then by  \cite[Theorem 3.3]{Seyed-chordal}, $\reg(I(G_i)^{(s)}) =2s + \nu(G_i)-1$. Therefore, by virtue of Theorem \ref{symRegLem}, we get $\reg\left(I(G)^{(s)}\right)=\max\left\{\reg\left(I(G_i)^{(s)}\right):1\leq i \leq r \right\}.$ Since for all $1\leq i\leq r$, $\reg(I(G_i)^{(s)})=2s+\reg(I(G_i))-2$, it follows from \cite[Proposition 3.1.2]{Am} that $$\reg\left(I(G)^{(s)}\right)=\max\left\{\reg\left(I(G_i)\right):1\leq i \leq r \right\}+2s-2=2s+\reg(I(G))-2.$$  Now, assume that each $G_i\in \mathcal{A}$. By \cite[Remark 4.10]{Se}, $\reg\left(I(G)^{(s)}\right)=\reg\left(I(G)^{s}\right)$ which completes the proof.
\end{proof}

In \cite{Se}, Selvaraja studied the regularity of powers of edge ideals of wheel graphs. In the following result, we prove Conjecture \ref{conj} for wheel graphs. 
\begin{corollary}
    Let $G=W_n$ be a wheel graph for $n\geq 4$. Then, $$\reg(I(G)^{(s)})=\reg(I(G)^s)=2s+\nu(C_n)-1 \text{ for all } s\geq 2.$$
\end{corollary}
\begin{proof}
    Since $G=W_n$ is  join of a vertex $v$ and a cycle $C_n$, by  Theorem \ref{symRegLem}, we have $\reg(I(G)^{(s)})=\max\{\reg I(C_n)^{(i)}-i+s:1\leq i\leq s\}.$ It follows from \cite[Theorem 5.3]{GHJsymbo} that for $ i\geq 2$, $\reg (I(C_n)^{(i)})=\reg(I(C_n)^i)=2i+\nu(C_n)-1$, and hence, our result follows from \cite[Theorem 4.7]{Se}.
\end{proof}

\section{Construction of Graphs}
In this section, we construct graphs whose edge ideals have different pair of algebraic invariants. Let $G$ and $H$ be graphs on vertex sets $V(G)=\{x_1,\dots, x_m\}$ and $V(H)=\{y_1,\dots, y_n\}$, respectively. One can see that if $C$ is a minimal vertex cover of $G*H$, then either $C=A\cup V(H)$ or  $C=V(G)\cup B$, where $A$ and $B$ are minimal vertex covers of $G$ and $H$, respectively. Therefore, we have the following short exact sequence
\begin{align}\label{ses}
0\rightarrow \frac{S}{I(G\ast H)}\rightarrow \frac{S}{I(G)+\m_2}\oplus \frac{S}{I(H)+\m_1}\rightarrow \sk \rightarrow 0,
\end{align}
where $\m_1=\langle x_1,\dots, x_m\rangle$ and $\m_2=\langle y_1,\dots, y_n\rangle$ and $S=\sk[x_1,\dots,x_m,y_1,\dots,y_n].$ Thus, by depth lemma,  $\depth(S/I(G\ast H))=1$.

Note that for a square-free monomial ideal $I$, by \cite[Observation 5.2]{KKSS}, $e(S/I)=|\Minh(I)|$, where $\Minh(I)=\{\mathfrak{p} \in \ass(S/I) : \hgt(\mathfrak{p})=\hgt(I)\}$. We now compute the multiplicity of join of graphs.

\begin{theorem}\label{JoinMultThm}
    Let $G$ and $H$ be graphs on vertex sets $V(G)=\{x_1,\ldots,x_m\}$ and $V(H)=\{y_1,\ldots,y_n\}$, respectively. Suppose $S=\sk[x_1,\dots,x_m,y_1,\dots,y_n]$. Then,     
    \[
    e(S/I(G\ast H))=\left\{ \begin{array}{cc}
    e(S/I(G))+e(S/I(H))&\text{ if } m+\hgt(I(H))=n+\hgt(I(G)),\\
    e(S/I(G))&\text{ if } m+\hgt(I(H))>n+\hgt(I(G)),\\
    e(S/I(H))&\text{ if } m+\hgt(I(H))<n+\hgt(I(G)).
    \end{array} \right.\]
\end{theorem}
\begin{proof}
    For a graph $G$, it is well known that minimal primes of $I(G)$ correspond to minimal vertex covers of $G$. Note that if $\mathfrak{p}\in \Minh(I(G\ast H))$, then either $\mathfrak{p}=\mathfrak{p}_1+\m_2$ or $\mathfrak{p}=\mathfrak{p}_2+\m_1$, where $\mathfrak{p}_1\in \Minh(I(G))$ and $\mathfrak{p}_2\in \Minh(I(H))$.
    Thus, we get
    \[
    |\Minh(I(G\ast H))|=\left\{ \begin{array}{cc}
    |\Minh(I(G))|+|\Minh(I(H))|&\text{ if } m+\hgt(I(H))=n+\hgt(I(G)),\\
    |\Minh(I(G))|&\text{ if } m+\hgt(I(H))>n+\hgt(I(G)),\\
    |\Minh(I(H))|&\text{ if } m+\hgt(I(H))<n+\hgt(I(G)).    \end{array} \right.\] Hence, the assertion follows.
\end{proof}
Now, we prove the algebraic properties of join of graphs which are used in rest of the section.
\begin{proposition}\label{selfjoin}
    Let $G^{*l}=G\ast \dots \ast G$ be join of $l$-copies of $G$, $S_G=\sk[V(G)]$ and $S=\sk[V(G^{*l})].$ Then, 
    \begin{enumerate}[\rm a)]
        \item $\reg(S/I(G^{*l}))=\reg(S_G/I(G))$,  
        \item $e(S/I(G^{*l}))=l\cdot e(S_G/I(G))$,
        \item $H(S/I(G^{*l}),t)=l\cdot H(S_G/I(G),t) -(l-1)$.
    \end{enumerate}
\end{proposition}
\begin{proof}
    (a)     and (c) follow from \cite[Proposition 3.12, Corollary 4.6]{Am}, respectively. (b) can be obtained by recursively applying Theorem \ref{JoinMultThm}. 
\end{proof}

Using above result, one can see that $e(S/I(K_n))=n$, and using \eqref{ses}, $\depth(S/I(K_n))=1$.
\begin{minipage}{\linewidth}
    \begin{minipage}{0.5\linewidth}
\noindent\begin{notation}            Let $K_n$ be the complete graph on the vertex set $\{x_1,\dots, x_n\}$. Then, for $1 \leq r \leq n$, we define a graph $W(n,r)$ on the vertex set $\{x_1,\dots,x_n,y_1,\dots,y_r\}$ with edge set $E(K_n)\cup\{\{x_i,y_i\}:1\leq i \leq r\}$. The graph shown in Figure \ref{W} is $W(5,3)$.
            \end{notation}
    \end{minipage}
    \begin{minipage}{0.55\linewidth}
        \hspace*{8mm}    
\begin{figure}[H]
 \begin{tikzpicture}[scale=1]
    \draw (-2,2)-- (0,2);
    \draw (-1,3)-- (-2,2);
    \draw (-2,2)-- (-2,1);
    \draw (-2,1)-- (0,1);
    \draw (0,1)-- (0,2);
    \draw (0,2)-- (-1,3);
    \draw (-1,3)-- (-2,1);
    \draw (-2,1)-- (0,2);
    \draw (-1,3)-- (0,1);
    \draw (0,1)-- (-2,2);
    \draw (-2,2)-- (-2,3);
    \draw (0,2)-- (0,3);
    \draw (-1,3)-- (-1,4);
    \begin{scriptsize}
    \fill  (-2,3) circle (1.5pt);
    \fill  (-1,3) circle (1.5pt);
    \fill  (0,1) circle (1.5pt);
    \fill  (-2,2) circle (1.5pt);
    \fill  (0,2) circle (1.5pt);
    \fill  (-2,1) circle (1.5pt);
    \fill  (0,3) circle (1.5pt);
    \fill  (-1,4) circle (1.5pt);
    \end{scriptsize}
    \end{tikzpicture}
    \caption{}\label{W}
   \end{figure}
    \end{minipage}
\end{minipage}
Now, we study the algebraic invariants of $W(n,r)$.
    \begin{theorem}\label{linearWhiskerLem}
        Let $G=W(n,r)$ and $S=\sk[V(G)]$. Then, we have the following:
        \begin{enumerate}[\rm a)]
            \item $I(G)$ has  linear resolution,
            \item $\pd(S/I(G))=n$ and $\depth(S/I(G))=r$,
            \item if $1 \leq r <n$, then $e(S/I(G))=n-r$.
        \end{enumerate}
    \end{theorem}
\begin{proof} 
(a) It is easy to see that $W(n,r)$ is a co-chordal graph. Hence, the assertion follows from \cite[Theorem 1]{Froberg}.
\par (b) We proceed by induction on $n$. Consider the following short exact sequence
\begin{equation*}
0\longrightarrow\dfrac{S}{I(G):x_1}\stackrel{\cdot x_1}{\longrightarrow}\dfrac{S}{I(G)}\longrightarrow\dfrac{S}{I(G)+\langle x_1 \rangle }\longrightarrow0.
\end{equation*}
Note that $I(G):x_1=\langle x_2,\ldots,x_n,y_1 \rangle $ and $ I(G)+\langle x_1 \rangle =I(W(n-1,r-1))+\langle x_1 \rangle $. Since $I(G):x_1$ is generated by a regular sequence of length $n$, $\pd( S/( I(G):x_1 ) )=n$ and $\depth(S/(I(G):x_1))=r$. Observe that $y_1$ is a regular element on $S/( I(G)+\langle x_1 \rangle) $. 
Thus, by induction, $\depth(S/( I(G)+\langle x_1 \rangle))=r$. Now, by applying depth lemma on the above short exact sequence, we have $\depth(S/I(G))=r$.
\par (c) We use induction on $n$. Let $A$ be a minimal vertex cover of $G$. Then, $|A \cap \{x_1,\ldots,x_n\}|\geq n-1$.  Also, $\{x_1,\ldots,x_{n-1}\}$ is a minimal vertex cover of $G$. Therefore, $\hgt(I(G))=n-1$, and hence, $\dim(S/I(G))=r+1$. From the proof of (b), we have  $\dim(S/( I(G):x_1))=r$ and $\dim(S/( I(G)+\langle x_1\rangle))=r+1$. Thus, by the above short exact sequence, $e(S/I(G))=e(S/( I(G)+\langle x_1\rangle))$. Hence, the assertion follows by induction. 
\end{proof}

In the following remark, we prove that $\reg(S/I) \leq \dim(S/I)$,  if $I$ is a square-free monomial ideal.  
\begin{remark}\label{key}
    It is a  known result, but for the sake of completeness, we prove it.
    Let $I \subset S=\sk[x_1,\dots,x_n]$ be a square-free monomial ideal.
    Let $\dim(S/I)=d $. Since $I$ is a square-free monomial ideal, there exists a simplicial complex $\Delta$ on $[n]$ such that the Stanley-Reisner ideal of $\Delta$ is $I$. It follows from \cite[Section 1.5]{Mon} that $\dim(\Delta)=d-1$. Therefore, $\tilde{H}_{l}(\Delta,\sk)=0$, for $l \geq d$. By Hochster's formula \cite[Theorem 8.1.1]{Mon}, $$\beta_{i,j}(S/I)=\sum_{A \subset [n], |A|=j} \dim_{\sk} \tilde{H}_{j-i-1}(\Delta_{A},\sk).$$  Let $\reg(S/I) =r$. Then, there exists $i,j$ such that $j-i=r$ and $\beta_{i,j}(S/I) \neq 0$.  Therefore, for some subset $A \subset [n]$, $\dim_{\sk} \tilde{H}_{r-1}(\Delta_A,\sk) \neq 0$, which implies that $r-1 \leq d-1$, and hence, $\reg(S/I) \leq \dim(S/I)$.
\end{remark}
As a consequence, we have the following:
\begin{corollary}\label{key1}
    Let $G$ be a triangle free graph  which is not a forest. Then, $\reg(S/I(G^c))=2$.
\end{corollary}
\begin{proof}
    It is a  well known fact that $\dim(S/I(G^c))$ is the maximum size of a maximal independent set of $G^c$ which is equal to the maximum size of a maximal clique in $G$. Since $G$ is triangle free, the maximum size of any maximal clique is $2$. Hence, by Remark \ref{key}, $\reg(S/I(G^c)) \leq 2$. Further, by \cite[Theorem 1]{Froberg}, $\reg(S/I(G^c))=2$ as $G$ is not chordal.
\end{proof}

\begin{remark}
    Let $G$ and $H$ be graphs with $\reg(I(G))=3$ and $\reg(I(H))\leq 3$. Then, by \cite[Proposition 3.12]{Am}, $\reg(I(G\ast H))=3$. For example, if $G^c$ is triangle free which is not a forest and $H$ is a bipartite graph whose edge ideal has regularity $3$ (\cite[Theorem 4.1]{Gim-14}) or co-chordal graph, then by Corollary  \ref{key1}, $\reg(I(G\ast H))=3$. In this way, one can construct graphs whose edge ideal have regularity $3$ which are neither bipartite graphs nor compliment of triangle free graphs.
\end{remark}

Let $F_n$ denote the bipartite graph on the vertex set $\{x_1,\dots, x_n,y_1,\dots, y_n\}$ with edge set $\{\{x_i,y_j\}: 1\leq i\leq j\leq n\}$. By virtue of \cite[Corollarly 9.1.14]{Mon}, $S/I(F_n)$ is Cohen-Macaulay. Since $\{x_1,\ldots,x_n\}$ is a minimal vertex cover of $F_n$,   $\hgt(I(F_n))=\dim(S/I(F_n))=n$. Observe that $\mu(I(F_n))=\binom{n+1}{2}$. Thus, it follows from  \cite[Theorem 4.3.7]{Vill} that $I(F_n)$ has linear resolution. The graph shown in Figure \ref{F4} is $F_4$.
\begin{figure}[H]
    \begin{tikzpicture}[scale=1]
    \draw (-2,3)-- (-2,1);
    \draw (-2,1)-- (-1,3);
    \draw (-2,1)-- (0,3);
    \draw (-2,1)-- (1,3);
    \draw (-1,3)-- (-1,1);
    \draw (-1,1)-- (0,3);
    \draw (0,3)-- (0,1);
    \draw (0,1)-- (1,3);
    \draw (1,3)-- (-1,1);
    \draw (1,3)-- (1,1);
    \begin{scriptsize}
    \fill  (-2,3) circle (1.5pt);
    \fill  (-2,1) circle (1.5pt);
    \fill  (-1,3) circle (1.5pt);
    \fill  (0,3) circle (1.5pt);
    \fill  (1,3) circle (1.5pt);
    \fill  (-1,1) circle (1.5pt);
    \fill  (0,1) circle (1.5pt);
    \fill  (1,1) circle (1.5pt);
    \end{scriptsize}
    \end{tikzpicture}
    \caption{}\label{F4}
\end{figure}

Now, for $1 \leq r \leq d$, we construct a graph $G$ with $\reg(S/I(G))=r$ and $\dim(S/I(G))=d$.
\begin{theorem}
    For $ 1 \leq r \leq d$, there exists a graph $G$ such that $\reg(S/I(G))=r$ and $\dim(S/I(G))=d$.    
\end{theorem}
\begin{proof}    
    Let $G$ be a disconnected graph with $r$ components such that the $r-1$ components are edges and one component is $F_{d-r+1}$. Then, $\reg(S/I(G))=r$ and $\dim(S/I(G))=d$.
\end{proof}
Note that $\{x_{2i}: 1 \leq i \leq n\}$ is a minimal vertex cover of both $P_{2n}$ and $P_{2n+1}$. Therefore, $\hgt(I(P_{2n})) \leq n$ and $\hgt(I(P_{2n+1})) \leq n$. Now, if possible, let $A$ be a minimal vertex cover of $P_{2n}$ such that $|A| <n$. Then, it is easy to see that $|\{e :  e \in E(P_{2n}) \text{ and } e \cap A \neq \emptyset\}| \leq 2n-2$, which is a contradiction. Similarly, one can prove that $\hgt(I(P_{2n+1}))=n$. Hence, $\dim(S/I(P_{2n}))=n$ and $\dim(S/I(P_{2n+1}))=n+1$.
\begin{lemma}\label{MultPath}
    Let $G=P_{2n+1}$ be a path graph. Then, $e(S/I(P_{2n+1}))=1$.
\end{lemma}
\begin{proof}
    We prove the assertion  by induction on $n$. If $n=1$, then $P_{3}$ has only one minimal vertex cover of minimal size. Hence, $e(S/I(P_{3})) =1$. Assume that $n>1$, and $e(S/I(P_{2n-1}))=1$. Now, consider the following short exact sequence,
    $$0 \to \frac{S}{I(P_{2n+1}):x_{2n+1}} \to \frac{S}{I(P_{2n+1})} \to \frac{S}{I(P_{2n})+\langle x_{2n+1}\rangle}    \to 0.$$
    Note that $I(P_{2n+1}):x_{2n+1} =I(P_{2n-1})+\langle x_{2n}\rangle$. Observe that $\dim(S/(I(P_{2n})+\langle x_{2n+1}\rangle))=n$,  $\dim(S/(I(P_{2n+1}):x_{2n+1}))=n+1$ and $\dim(S/I(P_{2n+1}))=n+1$.  Hence, $e(S/I(P_{2n+1}))=e(S/(I(P_{2n-1})+\langle x_{2n}\rangle))=1$.
\end{proof}

\begin{remark}
    Let $I \subseteq \m^2$ be a nonzero homogeneous ideal in a polynomial ring $S$ with $H(S/I, t)=(1+h_1t+h_2t^2+\dots+h_st^s)/(1-t)^d$, where $d=\dim(S/I)$. Then, $h_1=\codim(S/I)$. This implies that for any graph $G$ if $e(S/I(G))=1$, then $\deg(h_{S/I(G)}(t))\geq 2$.
    Note that if $I$ is a square-free monomial ideal, then by \cite[Theorem 5.1.7]{Bh1993}, $\deg(h_{S/I}(t))\leq \dim(S/I)$. Therefore, $e(S/I(G))=1$ implies that $\dim(S/I(G))\geq 2$. 
\end{remark}

We now prove that the multiplicity of edge ideal of a graph is not bounded by algebraic invariants such as regularity, depth, dimension, and degree of $h$-polynomial. We construct graphs with fixed multiplicity and one of the other invariants.

\begin{theorem}\label{main1}
    Let $e,r,s,\delta,d$ be positive integers. Then, we have the following:
    \begin{enumerate}[\rm a)]
        \item There exists a graph $G$ with $e(S/I(G))=e$ and $\reg(S/I(G))=r$.
        \item There exists a graph $G$ with $e(S/I(G))=e$ and $\deg (h_{S/I(G)}(t))=s$, provided $e\cdot s\geq 2$.
        \item There exists a graph $G$ with $e(S/I(G))=e$ and $\depth(S/I(G))=\delta$.
        \item There exists a graph $G$ with $e(S/I(G))=e$ and $\dim(S/I(G))=d$, provided $e\cdot d\geq 2$.
    \end{enumerate}
\end{theorem}
Observe that by Proposition \ref{selfjoin}, the regularity and dimension remain the same under self join operation but multiplicity increases. This observation is the key idea for the proof of Theorem \ref{main1}. 

\begin{proof}[Proof of Theorem \ref{main1}]
    (a) If $r$ is odd, then take $H=P_{3r}$, otherwise take $H=P_{3r+1}$. Then, by \cite[Theorem 4.7]{selvi_ha}, $\reg(S_H/I(H))=r$ and by Lemma \ref{MultPath}, $e(S_H/I(H))=1$, where $S_H=\sk[V(H)]$. Now, take $G=H^{\ast e}$ and $S=\sk[V(G)]$. Then, by Proposition \ref{selfjoin}, we have $\reg(S/I(G))=r$ and $e(S/I(G))=e$. 
    
    (b) and (d): For $d=s=1$, take $G=K_e$. Now,  we assume that $s\geq 2$, and $d\geq 2$. By short exact sequence \eqref{ses}, we get
    \begin{align}\label{hilbert-star}
    H\left(\frac{S_{K_{1,s}}}{I(K_{1,s})},t\right)=\frac{1}{(1-t)^s}+\frac{1}{(1-t)}-1,
    \end{align}
     where $S_{K_{1,s}}=\sk[V({K_{1,s}})].$      Take $G=K_{1,s}^{\ast e}$.
    Using Proposition \ref{selfjoin}, we have $$H\left(\frac{S}{I(G)},t\right)=\dfrac{e}{1-t}+\dfrac{e}{(1-t)^s}-(2e-1).$$ Thus,  $e(S/I(G))=e$ and $\deg(h_{S/I(G)}(t))=s=\dim(S/I(G))$.
    
    (c) For any $e$ and $\delta$, take $G=W(e+\delta,\delta)$. Then, it follows from Theorem \ref{linearWhiskerLem} that $e(S/I(G))=e$ and $\depth(S/I(G))=\delta$.
\end{proof}

We now prove that the depth of edge ideal of a graph is independent from regularity and degree of $h$-polynomial. Indeed, we construct graph whose edge ideal has depth $\delta$ and regularity $r$. Also, we show the existence of a graph whose edge ideal has depth $\delta $ and degree of $h$-polynomial $s$.

\begin{theorem}\label{main2}
    Let $\delta,r,s$ be positive integers. Then, we have the following:
    \begin{enumerate}[\rm a)]
        \item There exists a graph $G$ such that $\depth(S/I(G))=\delta$ and $\reg(S/I(G))=r$.
        \item There exists a graph $G$ such that $\depth(S/I(G))=\delta$ and $\deg (h_{S/I(G)}(t))=s$.
    \end{enumerate}
\end{theorem}
First, we prove the above theorem for $\delta=1$ or $r=1$ or $s=1$. Then, we use the following fact to complete our theorem: Let $I\subset R=\sk[x_1,\dots,x_m]$ and $J\subset T=\sk[y_1,\dots,y_n]$ be homogeneous ideals and $S=\sk[x_1,\dots,x_m,y_1,\dots,y_n]$. Then, it follows from \cite[Lemma 1,5]{HT} that $\deg(h_{S/(I+J)}(t))=\deg(h_{R/I}(t))+\deg(h_{T/J}(t))$. By  \cite[Lemma 2.2]{HT}, $\depth(S/(I+J))=\depth(R/I)+\depth(T/J)$ and by \cite[Lemma 3.2]{HT}, $\reg(S/(I+J))=\reg(R/I)+\reg(T/J)$. 
\begin{proof}[Proof of Theorem \ref{main2}]
    (a) \textbf{Case I}. For $\delta=1$ and $r \geq 1$, take $G=P_{3r}^{\ast 2}$. By Proposition \ref{selfjoin} and \cite[Theorem 4.7]{selvi_ha}, $\reg(S/I(G))=r$, and by the short exact sequence \eqref{ses}, $\depth(S/I(G))=1$.  
    \\ \textbf{Case II}. For  $2\leq \delta \leq r$, take $G$ to be the union of $P_{3(r-\delta+1)}^{\ast 2}$ and $\delta-1$ disjoint edges. Then, we get 
    $\depth(S/I(G))=\delta$ and $\reg(S/I(G))=r$.
    \\
    \textbf{Case III}. For $r=1$, consider $G=W(\delta+1,\delta)$. By Theorem \ref{linearWhiskerLem}, $\reg(S/I(G))=1$ and $\depth(S/I(G))=\delta$.
    \\ \textbf{Case IV}. For $2\leq r<\delta$, take $G$ to be the union of $W(\delta-r+2,\delta-r+1)$ and $r-1$ disjoint edges. This completes the proof of (a).
    \\ (b) \textbf{Case I}. For $\delta=1$, consider $G=K_{1,s}$. It follows from \eqref{ses} that $\depth(S/I(G))=1$ and by \eqref{hilbert-star}, $\deg (h_{S/I(G)}(t))=s$.
    \\ \textbf{Case II}. For $2\leq \delta \leq s$, take $G$ to be the union of $K_{1,s-\delta+1}$ and $\delta -1$ disjoint edges. Then, $\depth(S/I(G))=\delta$ and $\deg (h_{S/I(G)}(t))=s$.
    \\ \textbf{Case III}. For $s=1$, take $G=F_{\delta}$. Since $S/I(G)$ is Cohen-Macaulay of dimension $\delta$ and $\reg(S/I(G))=1$, we have $\depth(S/I(G))=\delta$ and $\deg (h_{S/I(G)}(t))=1$. 
    \\ \textbf{Case IV}. For $2\leq s< \delta$, take $G$ to be the union of $F_{\delta-s+1}$ and $s-1$ disjoint edges. Now, the assertion follows from Case III.
\end{proof}

\end{document}